\documentclass[10pt,a4paper]{article}
\usepackage[utf8]{inputenc}
\usepackage{amsmath}
\usepackage{amsfonts}
\usepackage{amssymb}
\usepackage{amsthm}
\usepackage{float}
\usepackage{aliascnt}
\usepackage{graphicx} 
\usepackage{enumitem}   
\newtheorem{lemma}{Lemma}

\newtheorem*{theorem}{Theorem}

\newtheorem*{remark}{Remark}

\title{ Solving a group equation  of length nine
}
\author{Muhammad Saeed Akram$^{1,*}$, Maira Amjid$^{2}$}

\date{$^{1}$Department of Mathematics, Ghazi University, Faculty of Science, Dera Ghazi Khanq 32200, Pakistan\\
$^{2}$Department of Mathematics, Khwaja Fareed University of
Engineering and Information Technology,
Rahim Yar Khan 64200, Pakistan\\
$^{*}$Correspondence: mrsaeedakram@gmail.com}
\begin{document}
\maketitle
\noindent \textbf{Abstract.}
 The Levin conjecture was proposed by Levin in 1962 which conjectures the solvability of any group equation with coefficients in a torsion free group. The Levin conjecture is recently shown to hold for group equations of length seven by weight test and curvature distribution. These methods are applied on a group equation of length nine to explore the validity of Levin conjecture. It is found that  the Levin's conjecture has an affirmative answer for this equation modulo some exceptional cases. 

%
\noindent \textbf{Keywords:}
Group equations; relative group presentations; asphericity; weight test; curvature distribution. \\
\textbf{2010 Mathematics Subject Classifications:} 20F05, 20E06, 57M05
\section{Introduction:}
Let $G$ be a non-trivial group and $t$ an element not in $G$. A \emph{group equation} or an \emph{equation} over $G$ is an equation of the form
$$s(t)=g_{1}t^{l_{1}}g_{2}t^{l_{2}}...g_{n}t^{l_{n}}=1 \quad (g_{i}\in G, l_{i}=\pm1)$$    
such that $l_{i}+l_{i+1}=0$ implies $g_{i+1}\neq 1 \in G$ (subscripts modulo $n$). The non-negative integer $n$ is known as the \emph{length} of equation $s(t)=1$. The equation $s(t)=1$ is said to be \emph{solvable} over $G$ if  $s(h)=1$ for some element $h$ of a group $H$ which contain $G$. Equivalently, $s(t)=1$ is solvable over $G$ if and only if the natural homomorphism from $G$ to $\frac{G * \langle t \rangle}{N}$ is injective, where $N$ is the normal closure of $s(t)=1$ in the free product $G * \langle t \rangle$. 
The equation $s(t)=1$ is called a \emph{singular} if $\sum^{n}_{i=1}{l_{i}}=0$ and \emph{non-singular} otherwise. 

The study of group equations was initiated by B. H. Neumann \cite{N} who solved an equation $t^{-1}g_1tg_2 =1$ over any torsion free group. Motivated by the solvability of the polynomial equations over fields, Levin \cite{L} studied the analogous problem for group equations and proved that the equation $s(t)=1$ ($l_{i}$ non-negative and not necessarily 1) is solvable over any group $G$, for $l_{1}+l_{2}+...+l_{n}=n$. These findings of Neumann and Levin gave the hope for the conjecture that any equation is solvable over a torsion free group, which is known as \emph{Levin's conjecture}. 
 
 There has been significant work to verify the Levin conjecture \cite{H,EH,E,IK} for group equations of length less than or equal to six. Recently, Mairaj and Edjvet \cite{BE} proved the Levin conjecture for all group equations of length seven 
 by using the weight test and curvature distribution method. By employing the methods used in \cite{BE}, Mairaj et al \cite{BAIA} have proved the conjecture  for a non-singular equation of length eight modulo one exceptional case.  The authors have done some significant work in \cite{ABA,ABI} using the weight test which  establishes the conjecture up to great extent for length eight. The equations of length nine are considered in \cite{B}, where it is proved that there are only three equations of length nine that are open. More recently, in \cite{FBA} and \cite{AAI}, the authors have investigated the conjecture for a non-singular equation of length nine (one of three) by applying these methods. For more recent work in this direction, see \cite{AH, ABA2, ABS}. In this paper, we consider the following group equation of length nine (one of the remaining two)
 $$s(t)=atbtctdtet^{-1}ftgthtit^{-1}=1$$
and investigate Levin's conjecture by applying the weight test and curvature distribution method. Here we prove Levin's conjecture for this equation modulo some exceptional cases.


\section{Methodology:}
\indent  A \emph{relative (group) presentation} is a presentation $\mathcal{P}=\langle G,x \mid r \rangle$ in which $r$ is a set of cyclically reduced words in $G * \langle x \rangle$. All the definitions concerning relative presentations can be seen in  \cite{BP}. 
Bogley and Pride established some sufficient conditions for the injectivity of natural map from $G$ to $\mathcal{P}=\langle G,x \mid r  \rangle$. They proved that the natural map from $G$ to $\mathcal{P}=\langle G,x \mid r  \rangle$ is injective if the relative presentation is orientable and aspherical \cite{BP}. In our case $s(t)=atbtctdtet^{-1}ftgthtit^{-1}$, therefore $x$ and $r$ consist of the single element $t$ and $s(t)$ respectively and so $\mathcal{P}=\langle G,t~|~s(t)\rangle$ is orientable. Therefore, in order to prove $s(t) = 1$ is solvable it is sufficient to prove asphericity of $\mathcal{P}$. Here we use weight test \cite{BP} and curvature distribution \cite{Ed} to establish asphericity of $\mathcal{P}$.

 All the necessary definitions concerning weight test can be found in \cite{BP}. The weight test states that if the star graph $\Gamma$ of $\mathcal{P}$ admits an aspherical weight function $\theta$, then $\mathcal{P}$ is aspherical \cite{BP}.  All the definitions related to pictures can be found in \cite{H}. The curvature distribution  asserts that if $K$ is a reduced picture over $\mathcal{P}$ then by Euler (or Gauss-Bonnet) formula, the sum of the curvature of all regions of $K$ is $4 \pi$, that is, $K$ contains regions of positive curvature \cite{EH}. Then, if for each region $\Delta$ of $K$ of positive curvature $c(\Delta)$, there is a neighbouring region $\widehat{\Delta}$, uniquely associated with $\Delta$, such that  $c(\widehat{\Delta}) + c(\Delta) \leq 0$, then the sum of the curvature of all regions of $K$ is non-positive, which implies that $\mathcal{P}$ is aspherical \cite{BE,B}.
 
Let $G$ be a torsion free group. By applying the transformation $u=tb$ on $s(t)=atbtctdtet^{-1}ftgthtit^{-1}=1$, it can be assumed that $b=1$. Recall that $\mathcal{P}=\langle G,t~|~s(t)\rangle$ where $$s(t)=atbtctdtet^{-1}ftgthtit^{-1} \quad (a,e,f,i \in G\setminus \{1\}, b=1,c,d,g,h \in G ).$$
 Furthermore, it can be assumed without any loss that $G$ is not cyclic and $G=\langle a,b,c,d,e,f,g,h,i \rangle$ \cite{H2}. Suppose
that $K$ is a reduced spherical diagram over $\mathcal{P}$. Up to
cyclic permutation and inversion, the regions of $K$ are given by $\Delta$ as shown in Figure \ref{1}(i). The star graph $\Gamma$ of $\mathcal{P}$ is given by Figure \ref{1}(ii).
\begin{figure}[H]
\centering
        \includegraphics[width=9cm]{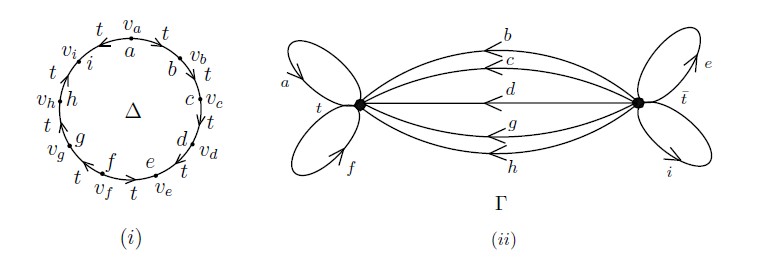}
        \caption{Region $\Delta$ of $K$ and star graph $\Gamma$ of $\mathcal{P}$ }
    \label{1}
\end{figure}
Looking at closed paths in star graph $\Gamma$, using the fact that $G$ is torsion free  and working modulo cyclic permutation and inversion, the possible labels of vertices of degree 2 for a region $\Delta$ of $K$ are 
$$ S=\lbrace af,af^{-1},ei,ei^{-1},hb^{-1},hc^{-1}, hd^{-1},hg^{-1},gb^{-1},gc^{-1},gd^{-1},db^{-1},dc^{-1},cb^{-1} \rbrace. $$
We can work modulo equivalence, that is, modulo $t \leftrightarrow t^{-1}$, cyclic permutation, inversion, and 
$$a \leftrightarrow e^{-1}, b \leftrightarrow d^{-1}c \leftrightarrow c^{-1}f \leftrightarrow i^{-1}, g \leftrightarrow h^{-1}.$$
We will proceed according to the number $N$ of labels in $S$ that are admissible \cite{BP} and classify the cases correspondingly \cite{BE}. The following remark substantially reduces the number of cases to be considered.

\begin{remark}
\normalfont
The following observations hold trivially.
\begin{enumerate}
\item If all the admissible cycles have length greater than $2$ in the region $\Delta$ then
$c(\Delta)\leq c(3, 3, 3, 3, 3, 3, 3, 3, 3) = -\pi$.
\item If $af$, $af^{-1}$ are admissible, then $f^{2}=1$, a contradiction.
\item If $ei$, $ei^{-1}$ are admissible, then $i^{2}=1$, a contradiction.
\item At most two of $af$, $af^{-1}$, $ei$, $ei^{-1}$ are admissible.
\item If any two of $hb^{-1}$, $hc^{-1}$, $cb^{-1}$ are admissible then, so is the third.
\item If any two of $hb^{-1}$, $hd^{-1}$, $db^{-1}$ are admissible then, so is the third. 
\item If any two of $hb^{-1}$, $hg^{-1}$, $gb^{-1}$ are admissible then, so is the third.  
\item If any two of $hc^{-1}$, $hd^{-1}$, $dc^{-1}$ are admissible then, so is the third. 
\item If any two of $hc^{-1}$, $hg^{-1}$, $gc^{-1}$ are admissible then, so is the third. 
\item If any two of $hd^{-1}$, $hg^{-1}$, $gd^{-1}$ are admissible then, so is the third. 
\item If any two of $gb^{-1}$, $gc^{-1}$, $cb^{-1}$ are admissible then, so is the third.
\item If any two of $gb^{-1}$, $gd^{-1}$, $db^{-1}$ are admissible then, so is the third.
\item If any two of $gc^{-1}$, $gd^{-1}$, $dc^{-1}$ are admissible then, so is the third. 
\item If any two of $dc^{-1}$, $db^{-1}$, $cb^{-1}$ are admissible then, so is the third.
\end{enumerate} 
\end{remark}

\section{Main Results:}
In what follows, the vertex labels correspond to the closed paths in the star graph $\Gamma$. The proof of the following Lemma \ref{wt} is an immediate application of the weight test \cite{BP}.
\begin{lemma}\label{wt} The presentation $\mathcal{P}=\langle G,t~|~s(t)\rangle$ aspherical if any one of the following holds:
\begin{enumerate}
\item $a=f^{-1}$;
\item $a=f^{-1}$ and $ R \in \lbrace ei,ei^{-1},hb^{-1},hc^{-1}, hd^{-1},hg^{-1},gb^{-1},gc^{-1},gd^{-1},db^{-1},dc^{-1},cb^{-1} \rbrace $; 
\item $a=f^{-1}, e=i^{-1}$ and $R \in \lbrace hb^{-1},hc^{-1}
 hd^{-1},hg^{-1},db^{-1},dc^{-1} \rbrace $;
\item $a=f^{-1}, e=i$ and $R \in \lbrace hb^{-1},hc^{-1}, hd^{-1},hg^{-1}gb^{-1},gc^{-1},gd^{-1},db^{-1},dc^{-1},cb^{-1} \rbrace $;
\item $a=f^{-1}, h=b$ and $R \in \lbrace gc^{-1},gd^{-1},dc^{-1} \rbrace $;
\item $a=f^{-1}, h=c$ and $R \in \lbrace gb^{-1},gd^{-1},db^{-1} \rbrace $;
\item $a=f^{-1}, h=d$ and $R \in \lbrace gb^{-1},gc^{-1},cb^{-1} \rbrace $;
\item $a=f^{-1}, h=g$ and $R \in \lbrace db^{-1},dc^{-1},cb^{-1}$;
\item $a=f^{-1}, g=b, d=c$;
\item $a=f^{-1}, g=c, d=b$;
\item $a=f^{-1}, g=d, c=b$;
\item $a=f^{-1}, e=i^{-1}, h=b$ and $R \in \lbrace gc^{-1},gd^{-1},dc^{-1} \rbrace $;
\item $a=f^{-1}, e=i^{-1}, h=c$ and $R \in \lbrace gb^{-1},db^{-1} \rbrace $;
\item $a=f^{-1}, e=i^{-1}, h=d$ and $R \in \lbrace gb^{-1},cb^{-1} \rbrace $;
\item $a=f^{-1}, e=i^{-1}, h=g$ and $R \in \lbrace db^{-1},dc^{-1} \rbrace $;
\item $a=f^{-1}, e=i, h=b$ and $R \in \lbrace gc^{-1},gd^{-1},dc^{-1} \rbrace $;
\item $a=f^{-1}, e=i, g=d$ and $R \in \lbrace hc^{-1},cb^{-1} \rbrace $;
\item $a=f^{-1}, e=i, h=c$ and $R \in \lbrace gb^{-1},db^{-1} \rbrace $;
\item $a=f^{-1}, e=i, g=c$ and $R \in \lbrace hd^{-1},db^{-1} \rbrace $;
\item $a=f^{-1}, e=i, h=d$ and $R \in \lbrace gb^{-1},cb^{-1} \rbrace $;
\item $a=f^{-1}, e=i, g=b, d=c$;
\item $a=f^{-1}, e=i, h=g$ and $R \in \lbrace db^{-1},dc^{-1},cb^{-1} \rbrace $;
\item $a=f^{-1}, h=b, h=c, c=b$;
\item $a=f^{-1}, h=b, h=d, d=b$;
\item $a=f^{-1}, h=b, h=g, g=b$;
\item $a=f^{-1}, h=c, h=d, d=c$;
\item $a=f^{-1}, h=c, h=g, g=c$;
\item $a=f^{-1}, d=c, d=b, c=b$;
\item $a=f^{-1}, g=d, g=c, d=c$;
\item $a=f^{-1}, g=d, g=b, d=b$;
\item $a=f^{-1}, g=d, h=g, h=d$;
\item $a=f^{-1}, g=c, g=b, c=b$;
\item $a=f^{-1}, e=i^{-1}, h=b, h=c, c=b$;
\item $a=f^{-1}, e=i^{-1}, h=b, h=d, d=b$;
\item $a=f^{-1}, e=i^{-1}, h=b, h=g, g=b$;
\item $a=f^{-1}, e=i^{-1}, h=c, h=d, d=c$;
\item $a=f^{-1}, e=i^{-1}, h=c, h=g, g=c$;
\item $a=f^{-1}, e=i^{-1}, d=c, d=b, c=b$;
\item $a=f^{-1}, e=i, h=b, h=c, c=b$;
\item $a=f^{-1}, e=i, h=b, h=d, d=b$;
\item $a=f^{-1}, e=i, h=b, h=g, g=b$;
\item $a=f^{-1}, e=i, g=d, g=c, d=c$;
\item $a=f^{-1}, e=i, g=d, g=b, d=b$;
\item $a=f^{-1}, e=i, g=d, h=g, h=d$;
\item $a=f^{-1}, e=i, h=c, h=d, d=c$;
\item $a=f^{-1}, e=i, h=c, h=g, g=c$;
\item $a=f^{-1}, e=i, g=c, g=b, c=b$;
\item $a=f^{-1}, e=i, d=c, d=b, c=b$;
\item $a=f^{-1}, h=b, h=c, c=b, g=d$;
\item $a=f^{-1}, h=b, h=d, d=b, g=c$;
\item $a=f^{-1}, h=b, h=g, g=b, d=c$;
\item $a=f^{-1}, h=c, h=d, d=c, g=b$;
\item $a=f^{-1}, h=c, h=g, g=c, d=b$;
\item $a=f^{-1}, d=c, d=b, c=b, h=g$;
\item $a=f^{-1}, g=d, g=c, d=c, h=b$;
\item $a=f^{-1}, g=d, g=b, d=b, h=c$;
\item $a=f^{-1}, g=d, h=g, h=d, c=b$;
\item $a=f^{-1}, g=c, g=b, c=b, h=d$;
\item $a=f^{-1}, e=i^{-1}, h=b, h=c, c=b, g=d$;
\item $a=f^{-1}, e=i^{-1}, h=b, h=d, d=b, g=c$;
\item $a=f^{-1}, e=i^{-1}, h=b, h=g, g=b, d=c$;
\item $a=f^{-1}, e=i^{-1}, h=c, h=d, d=c, g=b$;
\item $a=f^{-1}, e=i^{-1}, h=c, h=g, g=c, d=b$;
\item $a=f^{-1}, e=i^{-1}, d=c, d=b, c=b, h=g$;
\item $a=f^{-1}, e=i, h=b, h=c, c=b, g=d$;
\item $a=f^{-1}, e=i, h=b, h=d, d=b, g=c$;
\item $a=f^{-1}, e=i, h=b, h=g, g=b, d=c$;
\item $a=f^{-1}, e=i, h=c, h=d, d=c, g=b$;
\item $a=f^{-1}, e=i, h=c, h=g, g=c, d=b$;
\item $a=f^{-1}, e=i, d=c, d=b, c=b, h=g$;
\item $a=f^{-1}, e=i, g=d, g=c, d=c, h=b$;
\item $a=f^{-1}, e=i, g=d, g=b, d=b, h=c$;
\item $a=f^{-1}, e=i, g=d, h=g, h=d, c=b$;
\item $a=f^{-1}, e=i, g=c, g=b, c=b, h=d$;
\item $a=f^{-1}, h=b, h=c, c=b, h=d, d=b, d=c$;
\item $a=f^{-1}, h=b, h=d, d=b, h=g, g=b, g=d$;
\item $a=f^{-1}, h=b, h=c, c=b, h=g, g=b, g=c$;
\item $a=f^{-1}, g=d, g=c, d=c, g=b, d=b, c=b$;
\item $a=f^{-1}, g=d, g=c, d=c, h=g, h=d, h=c$;
\item $a=f^{-1}, e=i^{-1}, h=b, h=c, c=b, h=d, d=b, d=c$;
\item $a=f^{-1}, e=i^{-1}, h=b, h=d, d=b, h=g, g=b, g=d$;
\item $a=f^{-1}, e=i^{-1}, h=b, h=c, c=b, h=g, g=b, g=c$;
\item $a=f^{-1}, e=i, h=b, h=c, c=b, h=d, d=b, d=c$;
\item $a=f^{-1}, e=i, h=b, h=d, d=b, h=g, g=b, g=d$;
\item $a=f^{-1}, e=i, h=b, h=c, c=b, h=g, g=b, g=c$;
\item $a=f^{-1}, e=i, g=d, g=c, d=c, g=b, d=b, c=b$;
\item $a=f^{-1}, e=i, g=d, g=c, d=c, h=g, h=d, h=c$;
\item $a=f^{-1}, g=b, g=c, c=b, g=d, d=b, d=c, h=b, h=c, h=d, h=g$;
\item $a=f^{-1}, e=i^{-1}, g=b, g=c, c=b, g=d, d=b, d=c, h=b, h=c, h=d, h=g$;
\item $a=f^{-1}, e=i, g=b, g=c, c=b, g=d, d=b, d=c, h=b, h=c, h=d, h=g$.
\end{enumerate}
\end{lemma}
\begin{figure}[H]
\centering
        \includegraphics[width=5cm]{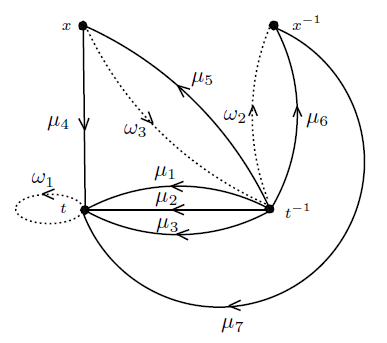}
        \caption{Star graph $\Gamma$}
    \label{wht}
\end{figure}
\begin{proof} We prove the Lemma for the case $a=f^{-1}$. The proof of the remaining cases follows similarly. The relator is $s(t)=atbtctdtet^{-1}a^{-1}tgthtit^{-1}$. We put $x=t^{-1}a^{-1}t$ to obtain $v_{1}=x^{-1}tctdtexgthti$ and  $v_{2}=x^{-1}t^{-1}a^{-1}t$.

The presentation $\mathcal{P}$ has star graph $\Gamma$ which is shown in Figure \ref{wht} in which $\mu_{1}= c$, $\mu_{2}=d$, $\mu_{3}=h$, $\mu_{4}=1$, $\mu_{5}=e$, $\mu_{6}=i$, $\mu_{7}=g$; and $\omega_{1}=a^{-1}$, $\omega_{2}=1$, $\omega_{3}=1$. We define a weight function $\theta$ such that  $\theta(\mu_{4})= \theta(\mu_{6})= \theta(\omega_{1})= \theta(\omega_{2})= 0$ and    
$\theta(\mu_{1})= \theta(\mu_{2})= \theta(\mu_{3})=\theta(\mu_{5}) = \theta(\mu_{7})= \theta(\omega_{3})=1$. Then $\Sigma(1-\theta(\mu_{i}))=\Sigma(1-\theta(\omega_{j}))=2$ indicates that the first condition of the weight test is fulfilled. Moreover, every cycle in $\Gamma$ of weight smaller than 2 has label $a^{m}$ or $i^{m}$, where $m \in \mathbb{Z}\setminus \{0\}$  and $a,i\in G\setminus \{1\}$, which implies $a$ and $i$ are torsion elements in $G$, a contradiction, so the second condition of weight test is fulfilled. Furthermore, since $\theta$ assigns non-negative weights to each edge, the third condition of weight test is obviously fulfilled.
\end{proof}
 From now onward, the label and the degree of a vertex $v$ of the region $\Delta$ will be denoted by $l_{\Delta}(v)$ and $d_{\Delta}(v)$  respectively. Furthermore, $l_{\Delta}\in\{ww_1,\dots,ww_k\}$ will be indicated by $l_{\Delta}(v)= \{ww_1,\dots,ww_k\}$.
 The forthcoming Lemmas and Corollary are proved by using curvature distribution \cite{Ed}.

\begin{lemma}\label{lem2} The presentation $\mathcal{P}=\langle G,t~|~s(t)\rangle$ is aspherical if 
$a=f,e=i$.
\end{lemma}
\begin{proof}
Here $l_{\Delta}(v_{a})=af^{-1}$, 
 $l_{\Delta}(v_{e})=ei^{-1}$, $l_{\Delta}(v_{f})=fa^{-1}$, $l_{\Delta}(v_{i})=ie^{-1}$, as given in Figure \ref{3}(i).
 \begin{figure}[H]
\centering
        \includegraphics[width=9cm]{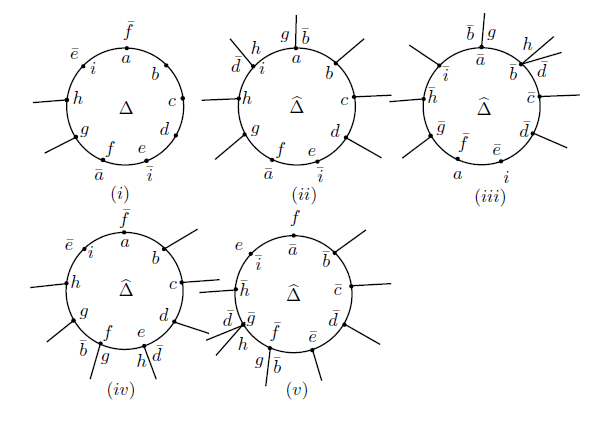}
        \caption{Regions $\Delta$ and $\widehat{\Delta}$}
    \label{3}
\end{figure} 
 
To obtain positive curvature the degree of the remaining vertices must be at least three. Suppose $d_{\Delta}(v_{g})=3$ and $d_{\Delta}(v_{h})=3$ or $d_{\Delta}(v_{h}) > 3$. Thus $l_{\Delta}(v_{g})=g\{a,a^{-1},f,f^{-1}\}b^{-1}$ which implies $l_{\Delta}(v_{h})=hd^{-1}\{i,b^{-1}, e, g^{-1}\}$. There are the following four cases to examine:
\begin{enumerate}
\item $l_{\Delta}(v_{g})=gab^{-1}$;
\item $l_{\Delta}(v_{g})=ga^{-1}b^{-1}$;
\item $l_{\Delta}(v_{g})=gfb^{-1}$;
\item $l_{\Delta}(v_{g})=gf^{-1}b^{-1}$.
\end{enumerate}
   
\begin{enumerate}
\item Since $l_{\Delta}(v_{g})=gab^{-1}$ so $l_{\Delta}(v_{h})=hd^{-1}i$. Add $c(\Delta)\leq \dfrac{\pi}{3}$ to $c(\widehat{\Delta})$ is given by Figure \ref{3}(ii). Notice that $d_{\widehat{\Delta}}(v_{e})=d_{\widehat{\Delta}}(v_{f})=2$ and all other vertices have degree at least 3. Therefore $c(\widehat{\Delta}) \leq c(2,2,3,3,3,3,3,3,3) = \dfrac{-\pi}{3}$.
\item Since $l_{\Delta}(v_{g})=ga^{-1}b^{-1}$ so $l_{\Delta}(v_{h})=hd^{-1}b^{-1}w$ where $w\in \{c,d,g,h\}$ which implies $l_{\Delta}(v_{h})> 3$. Add $c(\Delta)\leq \dfrac{\pi}{6}$ to $c(\widehat{\Delta})$ is given by Figure \ref{3}(iii). Observe that $d_{\widehat{\Delta}}(v_{f^{-1}})=d_{\widehat{\Delta}}(v_{e^{-1}})=2$. Notice that either $d_{\widehat{\Delta}}(v_{a^{-1}})=3$ or $d_{\widehat{\Delta}}(v_{i^{-1}})=2$, since $d_{\widehat{\Delta}}(v_{a^{-1}})=3$ already present so $d_{\widehat{\Delta}}(v_{i^{-1}})> 2$ otherwise contradiction occur and all other vertices have degree at least 3. Therefore $c(\widehat{\Delta}) \leq c(2,2,3,3,3,3,3,3,4) = \dfrac{-\pi}{2}$.
\item Since $l_{\Delta}(v_{g})=gfb^{-1}$ so $l_{\Delta}(v_{h})=hd^{-1}e$. Add $c(\Delta)\leq \dfrac{\pi}{3}$ to $c(\widehat{\Delta})$ is given by Figure \ref{3}(iv). Notice that  $d_{\widehat{\Delta}}(v_{i})=d_{\Delta}(v_{a})=2$ and all other vertices have degree at least 3. Therefore $c(\widehat{\Delta}) \leq c(2,2,3,3,3,3,3,3,3) = \dfrac{-\pi}{3}$.
\item Since $l_{\Delta}(v_{g})=gf^{-1}b^{-1}$ so $l_{\Delta}(v_{h})=hd^{-1}g^{-1}w$ where $w\in \{b,c,d,h\}$ which implies $l_{\Delta}(v_{h})> 3$. Add $c(\Delta)\leq \dfrac{\pi}{6}$ to $c(\widehat{\Delta})$ is given by Figure \ref{3}(v). Remark that $d_{\widehat{\Delta}}(v_{a^{-1}})=d_{\widehat{\Delta}}(v_{i^{-1}})=2$. Notice that either $d_{\widehat{\Delta}}(v_{f^{-1}})=3$ or $d_{\widehat{\Delta}}(v_{e^{-1}})=2$ since $d_{\widehat{\Delta}}(v_{f^{-1}})=3$ already present so $d_{\widehat{\Delta}}(v_{e^{-1}})> 2$ otherwise contradiction occur and all other vertices have degree at least 3. Therefore $c(\widehat{\Delta}) \leq c(2,2,3,3,3,3,3,3,4)= \dfrac{-\pi}{2}$.
\end{enumerate}
\end{proof}

\begin{lemma}\label{lem3} The presentation $\mathcal{P}=\langle G,t~|~s(t)\rangle$ aspherical if any one of the following holds:
\begin{enumerate}
\item $a=f$ and $R \in \{hb^{-1},hg^{-1},gd^{-1},gc^{-1},db^{-1},dc^{-1}
,cb^{-1}\}$;
\item $h=b$ and $R \in \{gd^{-1},dc^{-1}\}$;
\item $h=d$ and $R \in \{gb^{-1},cb^{-1}\}$;
\item $h=c,d=b$; 
\item $h=g,d=b$;
\item $a=f,g=c,d=b$;
\item $a=f,g=d,c=b$;
\item $h=b,h=d,d=b$;
\item $a=f,g=d,g=c,d=c$;
\item $h=b,h=c,c=b,h=d,d=b,d=c$;
\item $h=b,h=d,d=b,h=g,g=b,g=d$.
\end{enumerate}
\end{lemma}
\begin{proof} Here we prove the case $a=f$, $h=b$. The proofs of the remaining cases follows similarly. 
In this case $\Delta$ is given in Figure \ref{4}.
\begin{figure}[H]
\centering
        \includegraphics[width=9cm]{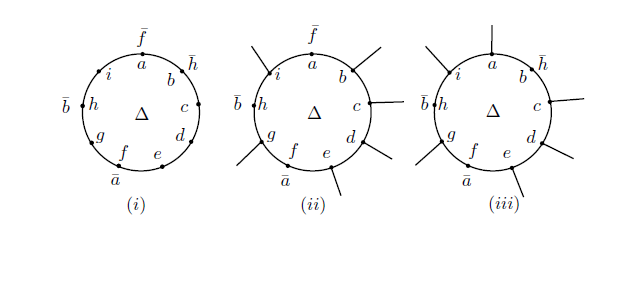}
        \caption{Region $\Delta$}
    \label{4}
\end{figure}
Since $d(v_{a})=d(v_{b})=2$ can not occur together thus there are the following two cases to examine:
\begin{enumerate}
\item[(a)] $d_{\Delta}(v_{a})=d_{\Delta}(v_{f})=d_{\Delta}(v_{h})=2$;
\item[(b)] $d_{\Delta}(v_{b})=d_{\Delta}(v_{f})=d_{\Delta}(v_{h})=2$.
\end{enumerate}
 Clearly, $c(\Delta) \leq 0$  for both these cases, as shown in Figures \ref{4}(ii) and \ref{4}(iii).
\end{proof}

 \begin{lemma}\label{lem4} The presentation $\mathcal{P}=\langle G,t~|~s(t)\rangle$ is aspherical if any one of the following holds:
\begin{enumerate}
\item $h=b,g=c$; 
\item $h=g,d=c$;
\item $h=b,h=c,c=b$;
\item $h=b,h=g,g=b$;
\item $h=c,h=d,d=c$;
\item $h=c,h=g,g=c$;
\item $d=c,d=b,c=b$.
\end{enumerate}
\end{lemma}
\begin{proof} Here we prove the case $h=b,g=c$. The proof of the remaining cases follows similarly. 
In this case $\Delta$ is given in Figure \ref{5}.
\begin{figure}[H]
\centering
        \includegraphics[width=4cm]{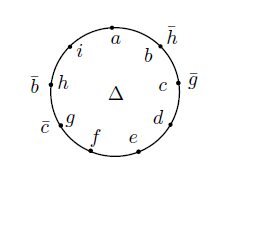}
        \caption{Region $\Delta$}
    \label{5}
\end{figure}
Since $d(v_{b})=d(v_{c})=2$ or $d(v_{g})=d(v_{h})=2$ can not occur together so  $c(\Delta) \leq 0$.
\end{proof}

\begin{lemma}\label{lem5} The presentation $\mathcal{P}=\langle G,t~|~s(t)\rangle$ is aspherical if any one of the following holds:\\
\begin{enumerate}
\item $a=f$ and $R \in \{hc^{-1},hd^{-1}\}$;
\item $a=f,h=b,g=c$;
\item $a=f,h=g,c=b$.
\end{enumerate}

\end{lemma}
\begin{proof}  Here we prove the case $a=f, h=c$. The proof of the remaining cases follows similarly.
In this case, $\Delta$ is given in Figure \ref{6}(i).
\begin{figure}[H]
\centering
        \includegraphics[width=6.5cm]{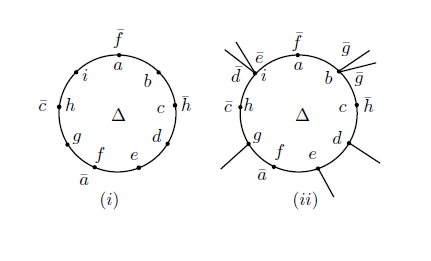}
        \caption{Region $\Delta$}
    \label{6}
\end{figure}
Here $d_{\Delta}(v_{a})=d_{\Delta}(v_{c})=d_{\Delta}(v_{f})=d_{\Delta}(v_{h})=2$ which implies $l_{\Delta}(v_{a})=af^{-1}$, $l_{\Delta}(v_{c})=ch^{-1}$, $l(v_{f})=fa^{-1}$, and $l_{\Delta}(v_{h})=hc^{-1}$ as shown in Figure \ref{6}(ii). To obtain positive curvature the degree of the remaining vertices must be at least three. Notice that $l_{\Delta}(v_{a})=af^{-1}$ and $l_{\Delta}(v_{c})=ch^{-1}$ implies that $l_{\Delta}(v_{b})=bg^{-1}g^{-1}w$ where $w \in \{a,a^{-1},f,f^{-1}\}$ which implies $ d_{\Delta}(v_{b}) > 3 $. Observe that $ l_{\Delta}(v_{a})=af^{-1} $ and $ l_{\Delta}(v_{h})=hc^{-1} $ implies that $ l_{\Delta}(v_{i})=ie^{-1}d^{-1}w $ where $ w \in \{b,c,g,h\} $ which implies $ d_{\Delta}(v_{i}) > 3 $. Since $ d_{\Delta}(v_{b}) > 3 $ and $ d_{\Delta}(v_{i}) > 3 $ so $ c(\Delta) \leq 0 $.
\end{proof}

\begin{lemma}\label{lem6} The presentation $\mathcal{P}=\langle G,t~|~s(t)\rangle$ is aspherical if any one of the following holds:\\
\begin{enumerate}
\item $a=f,h=b,g=d$;
\item $a=f,h=c$ and $R \in \{gd^{-1},db^{-1}\}$;
\item $a=f,h=g$ and $R \in \{db^{-1},dc^{-1}\}$.
\end{enumerate}
\end{lemma}
\begin{proof} Here we prove the case $a=f, h=b, g=d$. The proof of the remaining cases follows similarly.
In this case, $\Delta$ is given in Figure \ref{7}(i).\\
\begin{figure}[H]
\centering
        \includegraphics[width=9cm]{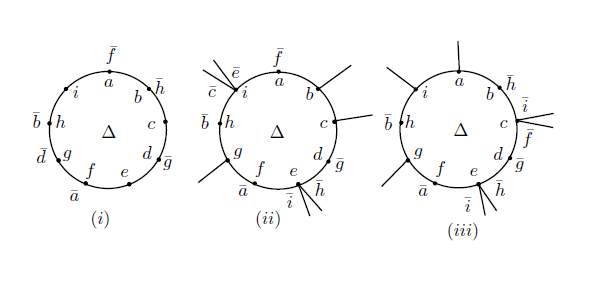}
        \caption{Region $\Delta$}
    \label{7}
\end{figure}

There are the following two cases to examine:
\begin{enumerate}
\item Here $d_{\Delta}(v_{a})=d_{\Delta}(v_{d})=d_{\Delta}(v_{f})=d_{\Delta}(v_{h})=2$ which implies $l_{\Delta}(v_{a})=af^{-1}$,$l_{\Delta}(v_{d})=dg^{-1}$,$l(v_{f})=fa^{-1}$, and $l_{\Delta}(v_{h})=hb^{-1}$ as shown in Figure \ref{7}(ii). To obtain positive curvature the degree of the remaining vertices must be at least three. Notice that $l_{\Delta}(v_{d})=dg^{-1}$ and $l_{\Delta}(v_{f})=fa^{-1}$ implies that $l_{\Delta}(v_{e})=ei^{-1}h^{-1}w$ where $w \in \{b,c,d,g\}$ which implies $ d_{\Delta}(v_{e}) > 3 $. Observe that $ l_{\Delta}(v_{a})=af^{-1} $ and $ l_{\Delta}(v_{h})=hb^{-1} $ implies that $ l_{\Delta}(v_{i})=ie^{-1}c^{-1}w $ where $ w \in \{b,d,g,h\} $ which implies $ d_{\Delta}(v_{i}) > 3 $. Since $ d_{\Delta}(v_{e}) > 3 $ and $ d_{\Delta}(v_{i}) > 3 $ so $ c(\Delta) \leq 0 $.
\item Here $d_{\Delta}(v_{b})=d_{\Delta}(v_{d})=d_{\Delta}(v_{f})=d_{\Delta}(v_{h})=2$ which implies $l_{\Delta}(v_{b})=bh^{-1}$,$l_{\Delta}(v_{d})=dg^{-1}$,$l(v_{f})=fa^{-1}$, and $l_{\Delta}(v_{h})=hb^{-1}$ as shown in Figure \ref{7}(iii). To obtain positive curvature the degree of the remaining vertices must be at least three. Notice that $l_{\Delta}(v_{b})=bh^{-1}$ and $l_{\Delta}(v_{d})=dg^{-1}$ implies that $l_{\Delta}(v_{c})=cf^{-1}i^{-1}w$ where $w \in \{b^{-1},d^{-1},g^{-1},h^{-1}\}$ which implies $ d_{\Delta}(v_{c}) > 3 $. Observe that $ l_{\Delta}(v_{d})=dg^{-1} $ and $ l_{\Delta}(v_{a})=af^{-1} $ implies that $ l_{\Delta}(v_{e})=ei^{-1}h^{-1}w $ where $ w \in \{b,c,d,g\} $ which implies $ d_{\Delta}(v_{e}) > 3 $. Since 
$ d_{\Delta}(v_{c}) > 3 $ and $ d_{\Delta}(v_{e}) > 3 $ so $ c(\Delta) \leq 0 $.
\end{enumerate}
\end{proof}

\begin{lemma}\label{7} The presentation $\mathcal{P}=\langle G,t~|~s(t)\rangle$ is aspherical if 
$a=f,g=b$.
\end{lemma}
\begin{proof} Here $l_{\Delta}(v_{a})=af^{-1}$, $l_{\Delta}(v_{b})=bg^{-1}$,
 $l_{\Delta}(v_{f})=fa^{-1}$, $l_{\Delta}(v_{g})=gb^{-1}$ as given in Figure \ref{8}(i).
 \begin{figure}[H]
\centering
        \includegraphics[width=9cm]{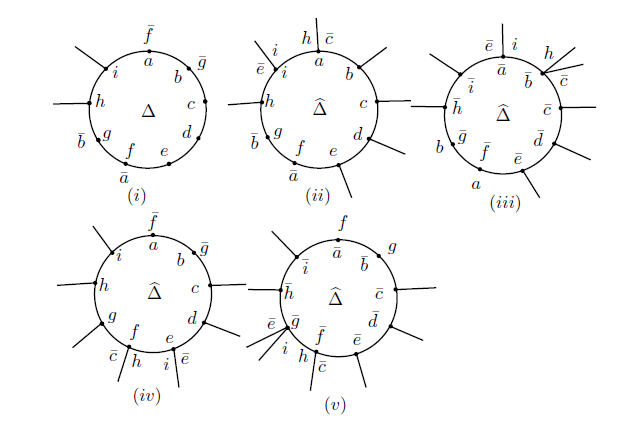}
        \caption{Regions $\Delta$ and $\widehat{\Delta}$}
    \label{8}
\end{figure} 
To obtain positive curvature the degree of the remaining vertices must be at least three. Suppose $d_{\Delta}(v_{h})=3$ and $d(v_{i})=3$ or $d(v_{i})> 3$. Thus $l_{\Delta}(v_{h})=h^{-1}\{a,a^{-1},f,f^{-1}\}c^{-1}$ which implies $l_{\Delta}(v_{i})=ie^{-1}\{i,b^{-1}, e, g^{-1}\}$. There are the following four cases to examine:
\begin{enumerate}
\item $l_{\Delta}(v_{h})=hac^{-1}$;
\item $l_{\Delta}(v_{h})=ha^{-1}c^{-1}$;
\item $l_{\Delta}(v_{h})=hfc^{-1}$;
\item $l_{\Delta}(v_{h})=hf^{-1}c^{-1}$.
\end{enumerate}
  
\begin{enumerate}
\item Since $l_{\Delta}(v_{h})=hac^{-1}$ so $l_{\Delta}(v_{i})=ie^{-1}i$. Add $c(\Delta)\leq \dfrac{\pi}{3}$ to $c(\widehat{\Delta})$ is given by Figure \ref{8}(ii). Observe that $d_{\widehat{\Delta}}(v_{f})=d_{\widehat{\Delta}}(v_{g})=2$. Notice that either $d_{\widehat{\Delta}}(v_{a})=3$ or $d_{\widehat{\Delta}}(v_{b})=2$ since $d_{\widehat{\Delta}}(v_{a})=3$ already present so $d_{\widehat{\Delta}}(v_{b})> 2$ otherwise contradiction occurs and all other vertices have degree at least 3. Therefore $c(\widehat{\Delta}) \leq c(2,2,3,3,3,3,3,3,3) = \dfrac{-\pi}{3}$.
\item Since $l_{\Delta}(v_{h})=ha^{-1}c^{-1}$ so $l_{\Delta}(v_{i})=ie^{-1}b^{-1}w$ where $w\in \{c,d,g,h\}$ which implies  $d_{\Delta}(v_{i})> 3$. Add $c(\Delta)\leq \dfrac{\pi}{6}$ to $c(\widehat{\Delta})$ is given by Figure \ref{8}(iii). Notice that $d_{\widehat{\Delta}}(v_{f^{-1}})=d_{\widehat{\Delta}}(v_{g^{-1}})=2$ and all other vertices have degree at least 3. Therefore $c(\widehat{\Delta}) \leq c(2,2,3,3,3,3,3,3,4) = \dfrac{-\pi}{2}$.
\item Since $l_{\Delta}(v_{h})=hfc^{-1}$ so $l_{\Delta}(v_{i})=ie^{-1}e$ yields a contradiction i.e $i=1$. Add $c(\Delta)\leq \dfrac{\pi}{6}$ to $c(\widehat{\Delta})$ is given by Figure \ref{8}(iv). Observe that $d_{\Delta}(v_{a})=d_{\Delta}(v_{b})=2$. Notice that either $d_{\widehat{\Delta}}(v_{f})=3$ or $d_{\widehat{\Delta}}(v_{g})=2$ since $d_{\widehat{\Delta}}(v_{f})=3$ already present so $d_{\widehat{\Delta}}(v_{g})> 2$ otherwise contradiction occurs and all other vertices have degree at least 3. Therefore $c(\widehat{\Delta}) \leq c(2,2,3,3,3,3,3,3,4) = \dfrac{-\pi}{2}$.
\item Since $l_{\Delta}(v_{h})=hf^{-1}c^{-1}$ so $l_{\Delta}(v_{i})=ie^{-1}g^{-1}w$ where $w\in \{b,c,d,h\}$ which implies  $d_{\Delta}(v_{i})> 3$. Add $c(\Delta)\leq \dfrac{\pi}{6}$ to $c(\widehat{\Delta})$ is given by Figure \ref{8}(v). Notice that $d_{\widehat{\Delta}}(v_{a^{-1}})=d_{\widehat{\Delta}}(v_{b^{-1}})=d_{\widehat{\Delta}}(v_{h^{-1}})=2$ and all other vertices have degree at least 3. Therefore $c(\widehat{\Delta}) \leq c(2,2,2,3,3,3,3,3,4)= \dfrac{-\pi}{6}$.
\end{enumerate}
\end{proof}
\begin{lemma} \label{lem8} The presentation $\mathcal{P}=\langle G,t~|~s(t)\rangle$ is aspherical if 
$h=c,g=b$.
\end{lemma}
\begin{proof}
Here $l_{\Delta}(v_{b})=bg^{-1}$, $l_{\Delta}(v_{c})=ch^{-1}$,
 $l_{\Delta}(v_{g})=gb^{-1}$, $l_{\Delta}(v_{h})=hc^{-1}$ as given in Figure \ref{9}(i).
To obtain positive curvature the degree of the remaining vertices must be at least three. Suppose $d_{\Delta}(v_{i})=3$ and $d_{\Delta}(v_{a})=3$ or $d_{\Delta}(v_{a}) > 3$. Thus $l_{\Delta}(v_{i})=i\{b,c,d,g,h\}d^{-1}$ which implies $l_{\Delta}(v_{a})=af^{-1}\{a, b, c, f, g\}$. There are the following five cases to examine:
\begin{enumerate}
\item $l_{\Delta}(v_{i})=ibd^{-1}$;
\item $l_{\Delta}(v_{i})=icd^{-1}$;
\item $l_{\Delta}(v_{i})=idd^{-1}$;
\item $l_{\Delta}(v_{i})=igd^{-1}$;
\item $l_{\Delta}(v_{i})=ihd^{-1}$.
\end{enumerate}
\begin{figure}[H]
\centering
        \includegraphics[width=9cm]{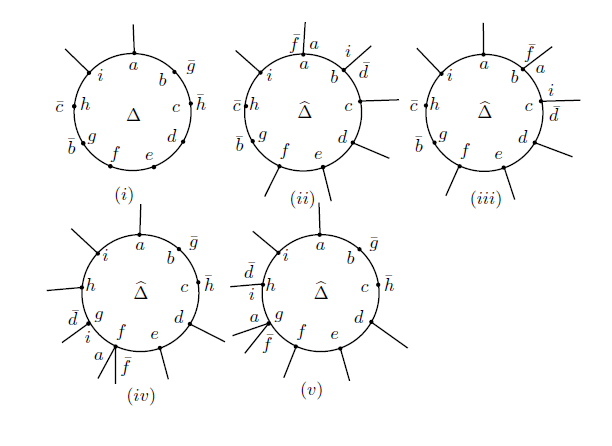}
       \caption{Regions $\Delta$ and $\widehat{\Delta}$}
    \label{9}
\end{figure}
 
\begin{enumerate}
\item Since $l_{\Delta}(v_{i})=ibd^{-1}$ so $l_{\Delta}(v_{a})=af^{-1}a$. Add $c(\Delta)\leq \dfrac{\pi}{3}$ to $c(\widehat{\Delta})$ is given by Figure \ref{9}(ii). Observe that $d_{\widehat{\Delta}}(v_{g})=d_{\widehat{\Delta}}(v_{h})=2$. Notice that either $d_{\widehat{\Delta}}(v_{b})=3$ or $d_{\widehat{\Delta}}(v_{c})=2$ since $d_{\widehat{\Delta}}(v_{b})=3$ already present so $d_{\widehat{\Delta}}(v_{c})> 2$ otherwise contradiction occurs and all other vertices have degree at least 3. Therefore $c(\widehat{\Delta}) \leq c(2,2,3,3,3,3,3,3,3) = \dfrac{-\pi}{3}$.
\item Since $l_{\Delta}(v_{i})=icd^{-1}$ so $l_{\Delta}(v_{a})=af^{-1}bw$ where $w\in \{c^{-1},d^{-1},g^{-1},h^{-1}\}$ which implies  $d_{\Delta}(v_{a})> 3$. Add $c(\Delta)\leq \dfrac{\pi}{6}$ to $c(\widehat{\Delta})$ is given by Figure \ref{9}(iii). Notice that $d_{\widehat{\Delta}}(v_{g})=d_{\widehat{\Delta}}(v_{h})=2$ and all other vertices have degree at least 3. Therefore $c(\widehat{\Delta}) \leq c(2,2,3,3,3,3,3,3,4) = \dfrac{-\pi}{2}$.
\item Since $l_{\Delta}(v_{i})=idd^{-1}$ yields a contradiction i.e $i=1$ so $d_{\Delta}(v_{i})> 3$ and $l_{\Delta}(v_{a})=af^{-1}bcw$ where $w\in \{b^{-1},d^{-1},g^{-1},h^{-1}\}$ which implies  $d_{\Delta}(v_{a})> 3$. Since 
$ d_{\Delta}(v_{i}) > 3 $ and $ d_{\Delta}(v_{a}) > 3 $ so $ c(\Delta) \leq 0 $.
\item Since $l_{\Delta}(v_{i})=igd^{-1}$ so $l_{\Delta}(v_{a})=af^{-1}f$ yields a contradiction i.e., $a=1$ so $d_{\Delta}(v_{a})> 3$. Add $c(\Delta)\leq \dfrac{\pi}{6}$ to $c(\widehat{\Delta})$ is given by Figure \ref{9}(iv). Observe that  $d_{\Delta}(v_{b})=d_{\Delta}(v_{c})=2$. Notice that either $d_{\widehat{\Delta}}(v_{g})=3$ or $d_{\widehat{\Delta}}(v_{h})=2$ since $d_{\widehat{\Delta}}(v_{g})=3$ already present so $d_{\widehat{\Delta}}(v_{h})> 2$ otherwise contradiction occurs and all other vertices have degree at least 3. Therefore $c(\widehat{\Delta}) \leq c(2,2,3,3,3,3,3,3,4) = \dfrac{-\pi}{2}$.
\item Since $l_{\Delta}(v_{i})=ihd^{-1}$ so $l_{\Delta}(v_{a})=af^{-1}gw$ where $w\in \{b^{-1},c^{-1},d^{-1},h^{-1}\}$ which implies  $d_{\Delta}(v_{a})> 3$. Add $c(\Delta)\leq \dfrac{\pi}{6}$ to $c(\widehat{\Delta})$ is given by Figure \ref{9}(v). Notice that $d_{\widehat{\Delta}}(v_{b})=d_{\widehat{\Delta}}(v_{c})=2$ and all other vertices have degree at least 3. Therefore $c(\widehat{\Delta}) \leq c(2,2,3,3,3,3,3,3,4)= \dfrac{-\pi}{2}$.
\end{enumerate}
\end{proof}
\begin{lemma}\label{lem9}  The presentation $\mathcal{P}=\langle G,t \mid r \rangle$  is aspherical if 
$a=f,e=i$ and $h=b$.
\end{lemma}
\begin{proof}
There are the following four cases to examine:
\begin{enumerate}
\item $d_{\Delta}(v_{i})=d_{\Delta}(v_{a})=d_{\Delta}(v_{e})=d_{\Delta}(v_{f})=2$;
\item $d_{\Delta}(v_{b})=d_{\Delta}(v_{e})=d_{\Delta}(v_{f})=d_{\Delta}(v_{h})=2$;
\item $d_{\Delta}(v_{b})=d_{\Delta}(v_{e})=d_{\Delta}(v_{f})=d_{\Delta}(v_{i})=2$;
\item $d_{\Delta}(v_{a})=d_{\Delta}(v_{e})=d_{\Delta}(v_{f})=d_{\Delta}(v_{h})=2$.
\end{enumerate}
\begin{figure}[H]
\centering
        \includegraphics[width=12cm]{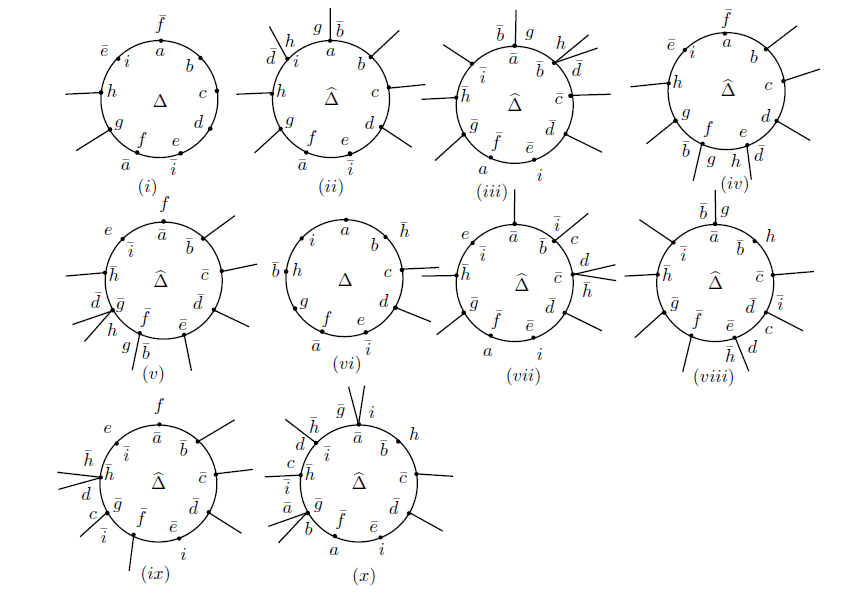}
        \caption{Regions $\Delta$ and $\widehat{\Delta}$}
    \label{10}
\end{figure}

\begin{enumerate}
\item Here $l_{\Delta}(v_{a})=af^{-1}$, $l_{\Delta}(v_{e})=ei^{-1}$, $l_{\Delta}(v_{f})=fa^{-1}$, $l_{\Delta}(v_{i})=ie^{-1}$, as given in Figure \ref{10}(i). To obtain positive curvature the degree of the remaining vertices must be at least three. Suppose $d_{\Delta}(v_{g})=3$ and $d_{\Delta}(v_{h})=3$ or $d_{\Delta}(v_{h})> 3$. Thus $l_{\Delta}(v_{g})=g\{a,a^{-1},f,f^{-1}\}b^{-1}$ which implies $l_{\Delta}(v_{h})=hd^{-1}\{i,b^{-1}, e, g^{-1}\}$. There are the following four cases to examine:
\begin{enumerate}
\item $l_{\Delta}(v_{g})=gab^{-1}$, $l_{\Delta}(v_{h})=hd^{-1}i$;
\item $l_{\Delta}(v_{g})=ga^{-1}b^{-1}$, $l_{\Delta}(v_{h})=hd^{-1}b^{-1}$;
\item $l_{\Delta}(v_{g})=gfb^{-1}$, $l_{\Delta}(v_{h})=hd^{-1}e$;
\item $l_{\Delta}(v_{g})=gf^{-1}b^{-1}$, $l_{\Delta}(v_{h})=hd^{-1}g^{-1}$.
\end{enumerate}
\begin{enumerate}
\item Since $l_{\Delta}(v_{g})=gab^{-1}$ so $l_{\Delta}(v_{h})=hd^{-1}i$. Add $c(\Delta)\leq \dfrac{\pi}{3}$ to $c(\widehat{\Delta})$ is given by Figure \ref{10}(ii). Notice that $d_{\widehat{\Delta}}(v_{e})=d_{\widehat{\Delta}}(v_{f})=2$ and all other vertices have degree at least 3. Therefore $c(\widehat{\Delta}) \leq c(2,2,3,3,3,3,3,3,3)= \dfrac{-\pi}{3}$.
\item Since $l_{\Delta}(v_{g})=ga^{-1}b^{-1}$ so $l_{\Delta}(v_{h})=hd^{-1}b^{-1}w$ where $w \in \{c,d,g,h\}$ which implies $d_{\Delta}(v_{h}) > 3$. Add $c(\Delta)\leq \dfrac{\pi}{6}$ to $c(\widehat{\Delta})$ is given by Figure \ref{10}(iii). Notice that, either $ d_{\widehat{\Delta}}(v_{i^{-1}})=2 $ or $ d_{\widehat{\Delta}}(v_{h^{-1}})=2 $ otherwise contradiction occurs. Observe that either $ d_{\widehat{\Delta}}(v_{a^{-1}})=3 $ or $ d_{\widehat{\Delta}}(v_{i^{-1}})=2$ since  $ d_{\widehat{\Delta}}(v_{a^{-1}})=3 $ already present so $ d_{\widehat{\Delta}}(v_{i^{-1}})> 2$ otherwise contradiction occur and all other vertices have  degree at least 3. Therefore $c(\widehat{\Delta}) \leq c(2,2,3,3,3,3,3,3,4) = \dfrac{-\pi}{2}$.
\item Since $l_{\Delta}(v_{g})=gfb^{-1}$ so $l_{\Delta}(v_{h})=hd^{-1}e$. Add $c(\Delta)\leq \dfrac{\pi}{3}$ to $c(\widehat{\Delta})$ is given by Figure \ref{10}(iv). Notice that either $ d_{\widehat{\Delta}}(v_{h})=2 $ or $ d_{\widehat{\Delta}}(v_{i})=2 $ otherwise contradiction occur. Similarly notice that either $ d_{\widehat{\Delta}}(v_{a})=2 $ or $ d_{\widehat{\Delta}}(v_{b})=2 $ otherwise contradiction occur, so we have $d_{\widehat{\Delta}}(v_{h})=d_{\widehat{\Delta}}(v_{a})=2$ and all other vertices have degree at least 3. Therefore  $c(\widehat{\Delta}) \leq c(2,2,3,3,3,3,3,3,3)= \dfrac{-\pi}{3}$.
\item Since $l_{\Delta}(v_{g})=gf^{-1}b^{-1}$ so $l_{\Delta}(v_{h})=hd^{-1}g^{-1}w$ where $w \in \{b,c,d,h\}$ which implies $d_{\Delta}(v_{h}) > 3$. Add $c(\Delta)\leq \dfrac{\pi}{6}$ to $c(\widehat{\Delta})$ is given by Figure \ref{10}(v). Notice that, either $ d_{\widehat{\Delta}}(v_{b^{-1}})=2 $ or $ d_{\widehat{\Delta}}(v_{a^{-1}})=2 $ otherwise contradiction occurs. Observe that either $ d_{\widehat{\Delta}}(v_{f^{-1}})=3 $ or $ d_{\widehat{\Delta}}(v_{e^{-1}})=2 $ since  $ d_{\widehat{\Delta}}(v_{f^{-1}})=3 $ already present so $ d_{\widehat{\Delta}}(v_{e^{-1}})> 2$ otherwise contradiction occur and all other vertices have degree at least 3. Therefore $c(\widehat{\Delta}) \leq c(2,2,3,3,3,3,3,3,4) = \dfrac{-\pi}{2}$.
\end{enumerate}
\item Here $l_{\Delta}(v_{b})=bh^{-1}$, $l_{\Delta}(v_{e})=ei^{-1}$,
 $l_{\Delta}(v_{f})=fa^{-1}$, $l_{\Delta}(v_{h})=hb^{-1}$ as given in Figure \ref{10}(vi). To obtain positive curvature the degree of the remaining vertices must be at least three. Suppose $d_{\Delta}(v_{c})=3$ and $d_{\Delta}(v_{d})=3$ or $d_{\Delta}(v_{d})>3$. Thus $l_{\Delta}(v_{c})=c\{b^{-1},c^{-1},d^{-1},g^{-1},h^{-1}\}i^{-1}$ which implies $l_{\Delta}(v_{d})=dh^{-1}\{c^{-1},d^{-1},e^{-1},h^{-1},i^{-1}\}$. There are the following five cases to examine:
\begin{enumerate}
\item $l_{\Delta}(v_{c})=cb^{-1}i^{-1}$, $l_{\Delta}(v_{d})=dh^{-1}c^{-1}$;
\item $l_{\Delta}(v_{c})=cc^{-1}i^{-1}$, $l_{\Delta}(v_{d})=dh^{-1}d^{-1}$;
\item $l_{\Delta}(v_{c})=cd^{-1}i^{-1}$, $l_{\Delta}(v_{d})=dh^{-1}e^{-1}$;
\item $l_{\Delta}(v_{c})=cg^{-1}i^{-1}$, $l_{\Delta}(v_{d})=dh^{-1}h^{-1}$;
\item $l_{\Delta}(v_{c})=ch^{-1}i^{-1}$, $l_{\Delta}(v_{d})=dh^{-1}i^{-1}$.
\end{enumerate}
\begin{enumerate}
\item  Since $l_{\Delta}(v_{c})=cb^{-1}i^{-1}$ so $l_{\Delta}(v_{d})=dh^{-1}c^{-1}w$ where $w \in \{b,d,g,h\}$ which implies $d_{\Delta}(v_{d}) > 3$. Add $c(\Delta)\leq \dfrac{\pi}{6}$ to $c(\widehat{\Delta})$ is given by Figure \ref{10}(vii). Observe that either $ d_{\widehat{\Delta}}(v_{h^{-1}})=2 $ or $ d_{\widehat{\Delta}}(v_{i^{-1}})=2 $ otherwise contradiction occurs. Notice that either $ d_{\widehat{\Delta}}(v_{b^{-1}})=3 $ or $ d_{\widehat{\Delta}}(v_{a^{-1}})=2 $ otherwise contradiction occur, since $ d_{\widehat{\Delta}}(v_{b^{-1}})=3 $ already present so $ d_{\widehat{\Delta}}(v_{a^{-1}})> 2 $. So, we have $d_{\widehat{\Delta}}(v_{i^{-1}})=d_{\widehat{\Delta}}(v_{e^{-1}})=d_{\widehat{\Delta}}(v_{f^{-1}})=2$ and all other vertices have at least degree 3. Therefore  $c(\widehat{\Delta}) \leq c(2,2,2,3,3,3,3,3,4) = \dfrac{-\pi}{6}$.
\item  Since $l_{\Delta}(v_{c})=cc^{-1}i^{-1}$ yields a contradiction i.e $i=1$ so $d_{\Delta}(v_{c}) > 3$ and $l_{\Delta}(v_{d})=dh^{-1}d^{-1}w$ where $w \in \{b,c,g,h\}$ which implies $d_{\Delta}(v_{d}) > 3$. Since 
$ d_{\Delta}(v_{c}) > 3 $ and $ d_{\Delta}(v_{d}) > 3 $ so $ c(\Delta) \leq 0 $. 
\item Since $l_{\Delta}(v_{c})=cd^{-1}i^{-1}$ so $l_{\Delta}(v_{d})=dh^{-1}e^{-1}$. Add $c(\Delta)\leq \dfrac{\pi}{3}$ to $c(\widehat{\Delta})$ is given by Figure \ref{10}(viii). Observe that either $ d_{\widehat{\Delta}}(v_{b^{-1}})=2 $ or $ d_{\widehat{\Delta}}(v_{a^{-1}})=2 $ otherwise contradiction occurs. Notice that either $ d_{\widehat{\Delta}}(v_{i^{-1}})=2 $ or $ d_{\widehat{\Delta}}(v_{h^{-1}})=2 $ otherwise contradiction occurs. Similarly notice that either $ d_{\widehat{\Delta}}(v_{e^{-1}})=3$ or $ d_{\widehat{\Delta}}( v_{f^{-1}})=2$ since $ d_{\widehat{\Delta}}(v_{e^{-1}})=3 $ already present so $ d_{\widehat{\Delta}}(v_{f^{-1}})> 2$ otherwise contradiction occurs and all other vertices have degree at least 3. Therefore  $c(\widehat{\Delta})  \leq c(2,2,2,3,3,3,3,4,4) = \dfrac{-\pi}{3}$.
\item Since $l_{\Delta}(v_{c})=cg^{-1}i^{-1}$ so $l_{\Delta}(v_{d})=dh^{-1}h^{-1}w$ where $w \in \{b,c,d,g\}$ which implies $d_{\Delta}(v_{d}) > 3$. Add $c(\Delta)\leq \dfrac{\pi}{6}$ to $c(\widehat{\Delta})$ is given by Figure \ref{10}(ix). Observe that either $ d_{\widehat{\Delta}}(v_{b^{-1}})=2 $ or $ d_{\widehat{\Delta}}(v_{a^{-1}})=2 $ otherwise contradiction occurs. Notice that either $ d_{\widehat{\Delta}}(v_{g^{-1}})=2 $ or $ d_{\widehat{\Delta}}(v_{f^{-1}})=3 $,since $ d_{\widehat{\Delta}}(v_{g^{-1}})=3 $ already present so $ d_{\widehat{\Delta}}(v_{f^{-1}}) > 2 $ otherwise contradiction occur. Similarly notice that $d_{\widehat{\Delta}}(v_{a^{-1}})=d_{\widehat{\Delta}}(v_{i^{-1}})=d_{\widehat{\Delta}}(v_{e^{-1}})=2$ and all other vertices have at least degree 3. Therefore $c(\widehat{\Delta}) \leq c(2,2,2,3,3,3,3,3,4) = \dfrac{-\pi}{6}$.
\item Since $l_{\Delta}(v_{c})=ch^{-1}i^{-1}$ so $l_{\Delta}(v_{d})=dh^{-1}i^{-1}$. Add $c(\Delta)\leq \dfrac{\pi}{3}$ to $c(\widehat{\Delta})$ is given by Figure \ref{10}(x). Observe that either $ d_{\widehat{\Delta}}(v_{a^{-1}})=2 $ or $ d_{\widehat{\Delta}}(v_{b^{-1}})=2 $ otherwise contradiction occur. Notice that $ d_{\widehat{\Delta}}(v_{i^{-1}})=3 $ or $ d_{\widehat{\Delta}}(v_{a^{-1}})=2 $ since  $ d_{\widehat{\Delta}}(v_{i^{-1}})=3 $ already present so $ d_{\widehat{\Delta}}(v_{a^{-1}})> 2$ otherwise contradiction occurs. Similarly notice that $l_{\widehat{\Delta}}(v_{h^{-1}})= h^{-1}i^{-1}c$ and $l_{\widehat{\Delta}}(v_{f^{-1}})= f^{-1}a$ implies that $l_{\widehat{\Delta}}(v_{g^{-1}})= g^{-1}ba^{-1}$ yields a contradiction i.e. $a^{2}=1$ which implies $d_{\widehat{\Delta}}(v_{g^{-1}})>3$. Observe that $l_{\widehat{\Delta}}(v_{i^{-1}})= i^{-1}dh^{-1}$ and $l_{\widehat{\Delta}}(v_{b^{-1}})= b^{-1}h$ implies that $l_{\widehat{\Delta}}(v_{a^{-1}})= a^{-1}g^{-1}iw$ where $w \in \{b,c,d,h\}$ which implies $d_{\widehat{\Delta}}(v_{a^{-1}})>3$. So, we have $d_{\widehat{\Delta}}(v_{b^{-1}})=d_{\widehat{\Delta}}(v_{e^{-1}})=d_{\widehat{\Delta}}(v_{f^{-1}})=2$  and all other vertices have at least degree 3. Therefore $c(\widehat{\Delta}) \leq c(2,2,2,3,3,3,3,3,4) = \dfrac{-\pi}{3}$.
\end{enumerate}
\item Here $l_{\Delta}(v_{b})=bh^{-1}$, $l_{\Delta}(v_{e})=ei^{-1}$,
 $l_{\Delta}(v_{f})=fa^{-1}$, $l_{\Delta}(v_{i})=ie^{-1}$ as given in Figure \ref{11}(i). 
 \begin{figure}[H]
\centering
        \includegraphics[width=12cm]{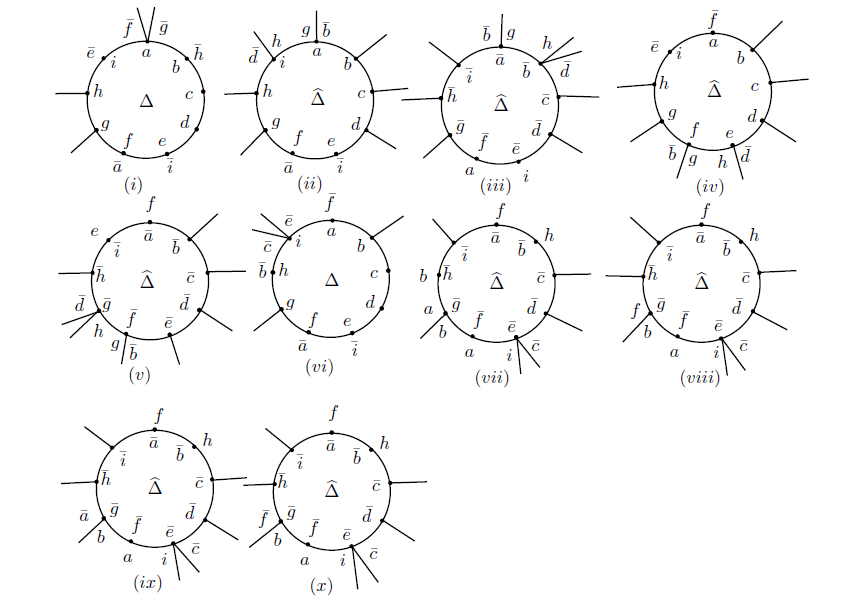}
        \caption{Regions $\Delta$ and $\widehat{\Delta}$}
    \label{11}
\end{figure}
To obtain positive curvature the degree of the remaining vertices must be at least three. Suppose $d_{\Delta}(v_{g})=3$ and $d_{\Delta}(v_{h})=3$ or $d_{\Delta}(v_{h})> 3$. Thus $l_{\Delta}(v_{g})=g\{a,a^{-1},f,f^{-1}\}b^{-1}$ which implies $l_{\Delta}(v_{h})=hd^{-1}\{i,b^{-1},e,g^{-1}\}$. Notice that $l_{\Delta}(v_{i})=ie^{-1}$ and $l_{\Delta}(v_{b})=bh^{-1}$ implies that $l_{\Delta}(v_{a})=ag^{-1}f^{-1}w$ where $w \in \{b,c,d,h\}$ which implies $d_{\Delta}(v_{a}) > 3$. There are the following four cases to examine:
 \begin{enumerate}
\item $l_{\Delta}(v_{h})=hd^{-1}i$;
\item $l_{\Delta}(v_{h})=hd^{-1}b^{-1}$;
\item $l_{\Delta}(v_{h})=hd^{-1}e$;
\item $l_{\Delta}(v_{h})=hd^{-1}g^{-1}$.
\end{enumerate}
\begin{enumerate}
\item  Since $d_{\Delta}(v_{h})=3$ and $d_{\Delta}(v_{a})> 3$. Add $c(\Delta)\leq \dfrac{\pi}{6}$ to $c(\widehat{\Delta})$ is given by Figure \ref{11}(ii). Notice that either $ d_{\widehat{\Delta}}(v_{h})=2 $ or $ d_{\widehat{\Delta}}(v_{i})=2 $ otherwise contradiction occurs. Observe that either $ d_{\widehat{\Delta}}(v_{a})=2 $ or $ d_{\widehat{\Delta}}(v_{b})=2 $ otherwise contradiction occur and all other vertices have at least degree 3. Therefore  $c(\widehat{\Delta}) \leq c(2,2,2,3,3,3,3,3,4) = \dfrac{-\pi}{6}$.
\item Since $l_{\Delta}(v_{h})=hd^{-1}b^{-1}w$ where $w \in \{c,d,g,h\}$ which implies $d_{\Delta}(v_{h}) > 3$. Notice that $l_{\Delta}(v_{i})=ie^{-1}$ and $l_{\Delta}(v_{b})=bh^{-1}$ as shown in Figure \ref{11}(iii) which implies that $l_{\Delta}(v_{a})=aa^{-1}f^{-1}$ yields a contradiction i.e., $f=1$ so $d_{\Delta}(v_{a})> 3$. Since $d_{\Delta}(v_{a})> 3$ and $d_{\Delta}(v_{h})> 3$ so $c(\Delta)\leq 0$.
\item Since $l_{\Delta}(v_{h})=hd^{-1}e$ and $d_{\Delta}(v_{a})> 3$. Add $c(\Delta)\leq \dfrac{\pi}{6}$ to $c(\widehat{\Delta})$ is given by Figure \ref{11}(iv). Notice that either $ d_{\widehat{\Delta}}(v_{h})=2 $ or $ d_{\widehat{\Delta}}(v_{i})=2 $ otherwise contradiction occurs. Observe that either $ d_{\widehat{\Delta}}(v_{a})=2 $ or $ d_{\widehat{\Delta}}(v_{b})=2 $ otherwise contradiction occurs and all other vertices have at least degree 3. Therefore  $c(\widehat{\Delta}) \leq c(2,2,2,3,3,3,3,3,4) = \dfrac{-\pi}{6}$.
\item Since $l_{\Delta}(v_{h})=hd^{-1}g^{-1}w$ where $w \in \{b,c,d,h\}$ which implies $d_{\Delta}(v_{h}) > 3$.  Notice that $l_{\Delta}(v_{i})=ie^{-1}$ and $l_{\Delta}(v_{b})=bh^{-1}$ as shown in Figure \ref{11}(v) which implies that $l_{\Delta}(v_{a})=aa^{-1}f^{-1}$ which yields a contradiction i.e., $f=1$ so $d_{\Delta}(v_{a})> 3$. Since $d_{\Delta}(v_{a})> 3$ and $d_{\Delta}(v_{h})> 3$ so $c(\Delta)\leq 0$.
\end{enumerate}
\item Here $l_{\Delta}(v_{a})=af^{-1}$, $l_{\Delta}(v_{e})=ei^{-1}$,
 $l_{\Delta}(v_{f})=fa^{-1}$, $l_{\Delta}(v_{h})=hb^{-1}$ as given in Figure \ref{11}(vi). To obtain positive curvature the degree of the remaining vertices must be at least three. Suppose $d_{\Delta}(v_{a})=2$ and $d_{\Delta}(v_{b})=3$ or $d_{\Delta}(v_{b})> 3$. Thus $l_{\Delta}(v_{b})=b\{a,a^{-1},f,f^{-1}\}g^{-1}$. Notice that $l_{\Delta}(v_{a})=af^{-1}$ and $l_{\Delta}(v_{b})=bh^{-1}$ implies that $l_{\Delta}(v_{i})=ic^{-1}e^{-1}w$ where $w\in \{b,d,g,h\} $ which implies $d_{\Delta}(v_{i})> 3$ but  $l_{\Delta}(v_{g})=ga^{-1}b^{-1}$ so there are the following four cases to examine:
\begin{enumerate}
\item $l_{\Delta}(v_{b})=bag^{-1}a$;
\item $l_{\Delta}(v_{b})=ba^{-1}g^{-1}$;
\item $l_{\Delta}(v_{b})=bfg^{-1}$;
\item $l_{\Delta}(v_{b})=bf^{-1}g^{-1}$.
\end{enumerate}
\begin{enumerate}
\item  Since $l_{\Delta}(v_{b})=bag^{-1}$. So $d_{\Delta}(v_{b})=3$ and $d_{\Delta}(v_{i})> 3$. Add $c(\Delta)\leq \dfrac{\pi}{6}$ to $c(\widehat{\Delta})$ is given by Figure \ref{11}(vii). Notice that either $ d_{\widehat{\Delta}}(v_{h^{-1}})=2 $ or $ d_{\widehat{\Delta}}(v_{i^{-1}})=2 $ otherwise contradiction occurs. Observe that either $ d_{\widehat{\Delta}}(v_{a^{-1}})=2 $ or $ d_{\widehat{\Delta}}(v_{b^{-1}})=2 $ otherwise contradiction occurs and all other vertices have at least degree 3. Therefore  $c(\widehat{\Delta}) \leq c(2,2,2,3,3,3,3,3,4) = \dfrac{-\pi}{6}$.
\item Since $l_{\Delta}(v_{b})=ba^{-1}g^{-1}$   but $l_{\Delta}(v_{g})=ga^{-1}b^{-1}$ yields a contradiction i.e., $g^{2}=1$. Notice that $l_{\Delta}(v_{a})=af^{-1}$ and $l_{\Delta}(v_{h})=hb^{-1}$ as shown in Figure \ref{11}(viii) which implies that $l_{\Delta}(v_{i})=ie^{-1}c^{-1}w$ where $w\in \{b,d,g,h\} $ implies that $d_{\Delta}(v_{i})> 3$. Since $d_{\Delta}(v_{b})> 3$ and $d_{\Delta}(v_{i})> 3$ so $c(\Delta)\leq 0$.
\item Since $l_{\Delta}(v_{h})=bfg^{-1}$ and $d_{\Delta}(v_{i})> 3$. Add $c(\Delta)\leq \dfrac{\pi}{6}$ to $c(\widehat{\Delta})$ is given by Figure \ref{11}(ix). Notice that either $ d_{\widehat{\Delta}}(v_{a^{-1}})=2 $ or $ d_{\widehat{\Delta}}(v_{b^{-1}})=2 $ otherwise contradiction occurs and all other vertices have degree at least 3. Therefore  $c(\widehat{\Delta}) \leq c(2,2,2,3,3,3,3,3,4) = \dfrac{-\pi}{6}$.
\item Since $l_{\Delta}(v_{b})=bf^{-1}g^{-1}$   but $l_{\Delta}(v_{g})=ga^{-1}b^{-1}$ yields a contradiction i.e $g^{2}=1$ which implies $d_{\Delta}(v_{i})> 3$. Notice that $l_{\Delta}(v_{a})=af^{-1}$ and $l_{\Delta}(v_{h})=hb^{-1}$ as shown in Figure \ref{11}(x) implies that $l_{\Delta}(v_{i})=ie^{-1}c^{-1}w$ where $w\in \{b,d,g,h\} $ which implies $d_{\Delta}(v_{i})> 3$. Since $d_{\Delta}(v_{b})> 3$ and $d_{\Delta}(v_{i})> 3$ so $c(\Delta)\leq 0$.
\end{enumerate}
\end{enumerate}
\end{proof}

\begin{lemma}\label{lem10}  The presentation $\mathcal{P}=\langle G,t~|~s(t)\rangle$  is aspherical if
 $a=f,h=b$ and $d=c$.
\end{lemma}
\begin{proof} As can be seen in Figure \ref{12}(i),
there are the following two cases to examine:
\begin{enumerate}
\item $d_{\Delta}(v_{a})=d_{\Delta}(v_{c})=d_{\Delta}(v_{f})=d_{\Delta}(v_{h})=2$;
\item $d_{\Delta}(v_{b})=d_{\Delta}(v_{d})=d_{\Delta}(v_{f})=d_{\Delta}(v_{h})=2$.
\end{enumerate}
\begin{figure}[H]
\centering
        \includegraphics[width=12cm]{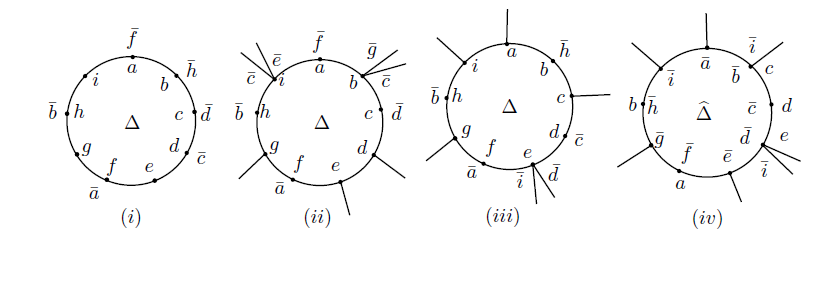}
        \caption{Regions $\Delta$ and $\widehat{\Delta}$}
    \label{12}
\end{figure}
\begin{enumerate}
\item Here $l_{\Delta}(v_{a})=af^{-1}$, $l_{\Delta}(v_{c})=cd^{-1}$, $l_{\Delta}(v_{f})=fa^{-1}$, $l_{\Delta}(v_{h})=hb^{-1}$, as given in Figure \ref{12}(ii). Notice that $l_{\Delta}(v_{a})=af^{-1}$ and $l_{\Delta}(v_{c})=cd^{-1}$ implies that $l_{\Delta}(v_{b})=bc^{-1}g^{-1}w$ where $w \in \{b,c,d,h\}$ which implies $ d_{\Delta}(v_{b}) > 3 $. Observe that $ l_{\Delta}(v_{a})=af^{-1} $ and $ l_{\Delta}(v_{h})=hb^{-1} $ implies that $ l_{\Delta}(v_{i})=ie^{-1}c^{-1}w $ where $ w \in \{b,d,g,h\} $ which implies $ d_{\Delta}(v_{i}) > 3 $. Since 
$ d_{\Delta}(v_{b}) > 3 $ and $ d_{\Delta}(v_{i}) > 3 $ so $ c(\Delta) \leq 0 $. 
\item Here $l_{\Delta}(v_{b})=bh^{-1}$, $l_{\Delta}(v_{d})=dc^{-1}$, $l_{\Delta}(v_{f})=fa^{-1}$, $l_{\Delta}(v_{h})=hb^{-1}$, as given in Figure \ref{12}(iii). Notice that  $l_{\Delta}(v_{a})=af^{-1}$ and $l_{\Delta}(v_{d})=dc^{-1}$ implies that $l_{\Delta}(v_{e})=ei^{-1}d^{-1}w$ where $w \in \{b,c,g,h\}$ which implies $ d_{\Delta}(v_{e}) > 3 $. Add $c(\Delta)\leq \dfrac{\pi}{6}$ to $c(\widehat{\Delta})$ is given by Figure \ref{12}(iv). Observe that $d_{\widehat{\Delta}}(v_{f^{-1}})=d_{\widehat{\Delta}}(v_{h^{-1}})=d_{\widehat{\Delta}}(v_{c^{-1}})=2$ and all other vertices have degree at least 3. Therefore $c(\widehat{\Delta}) \leq c(2,2,2,3,3,3,3,3,4) = \dfrac{-\pi}{6}$.
\end{enumerate}
\end{proof}
\begin{lemma}\label{lem11} The presentation $\mathcal{P}=\langle G,t~|~s(t)\rangle$  is aspherical if any one of the following holds:
\begin{enumerate}
\item $a=f,e=i$ and $h=g$;
\item $a=f,e=i$ and $d=b$;
\item $a=f,h=d$ and $c=b$;
\item $a=f,g=b$ and $d=c$;
\item $a=f,e=i,h=b$ and $g=d$;
\item $a=f,e=i,h=g$ and $d=b$;
\item $a=f,h=b,h=c$ and $c=b$;
\item $a=f,h=b,h=d$ and $d=b$;
\item $a=f,h=b,h=g$ and $g=b$;
\item $a=f,h=c,h=d$ and $d=c$;
\item $a=f,h=c,h=g$ and $g=c$;
\item $a=f,d=c,d=b$ and $c=b$;
\item $a=f,g=d,h=g$ and $h=d$;
\item $a=f,g=c,g=b$ and $c=b$;
\item $a=f,h=b,h=c,c=b$ and $g=d$;
\item $a=f,h=b,h=g,g=b$ and $d=c$;
\item $a=f,h=c,h=g,g=c$ and $d=b$;
\item $a=f,d=c,d=b,c=b$ and $h=g$;
\item $a=f,g=d,g=c,d=c$ and $h=b$;
\item $a=f,g=d,h=g,h=d$ and $c=b$;
\item $h=b,h=c,c=b,h=g,g=b$ and $g=c$;
\item $a=f,h=b,h=c,c=b,h=d,d=b$ and $d=c$;
\item $g=b,g=c,c=b,g=d,d=b,d=c,h=b,h=c,h=d$ and $h=g$.
\end{enumerate}
\end{lemma}
\begin{proof} The proof of these cases are similar to the proofs of Lemmas \ref{lem2}-\ref{lem10}. 
\end{proof}
Now we state the main result of the paper.
\begin{theorem}\label{th}
The presentation $\mathcal{P}=\langle G,t~|~s(t)\rangle$ where $s(t)=atbtctdtet^{-1}ftgthtit^{-1}$ is aspherical modulo some exceptional cases which are given below:
\end{theorem}
\begin{enumerate}
\item $a=f, e=i$ and $R \in\{hc^{-1},hd^{-1},dc^{-1}\}$;
\item $a=f, h=d$ and $R \in\{gb^{-1},gc^{-1}\}$;
\item $a=f, h=c, g=b$;
\item $a=f, e=i, h=b$ and $R \in \{gc^{-1},dc^{-1}\}$;
\item $a=f, e=i, h=c$ and $R \in \{gb^{-1},db^{-1}\}$;
\item $a=f, e=i, h=d$ and $R \in \{gb^{-1},cb^{-1}\}$;
\item $a=f, e=i, h=g, d=c$;
\item $a=f, g=d, g=b, d=b$; 
\item $a=f, e=i, h=b, h=c, c=b$;
\item $a=f, e=i, h=b, h=d, d=b$;
\item $a=f, e=i, h=b, h=g, g=b$;
\item $a=f, e=i, h=c, h=d, d=c$;
\item $a=f, e=i, h=c, h=g, g=c$;
\item $a=f, e=i, d=c, d=b, c=b$;
\item $a=f, h=b, h=d, d=b, g=c$; 
\item $a=f, h=c, h=d, d=c, g=b$; 
\item $a=f, g=d, g=b, d=b, h=c$;
\item $a=f, g=c, g=b, c=b, h=d$;  
\item $a=f, e=i, h=b, h=c, c=b, g=d$;
\item $a=f, e=i, h=b, h=d, d=b, g=c$;
\item $a=f, e=i, h=b, h=g, g=b, d=c$;
\item $a=f, e=i, h=c, h=d, d=c, g=b$;
\item $a=f, e=i, h=c, h=g, g=c, d=b$;
\item $a=f, e=i, d=c, d=b, c=b, h=g$;
\item $a=f, h=b, h=d, d=b, h=g, g=b, g=d$;
\item $a=f, h=b, h=c, c=b, h=g, g=b, g=c$;
\item $a=f, g=d, g=c, d=c, g=b, d=b, c=b$;
\item $a=f, g=d, g=c, d=c, h=g, h=d, h=c$;
\item $a=f, e=i, h=b, h=c, c=b, h=d, d=b, d=c$;
\item $a=f, e=i, h=b, h=d, d=b, h=g, g=b, g=d$;
\item $a=f, e=i, h=b, h=c, c=b, h=g, g=b, g=c$;
\item $a=f, g=b, g=c, c=b, g=d, d=b, d=c, h=b, h=c, h=d, h=g$;
\item $a=f, e=i, g=b, g=c, c=b, g=d, d=b, d=c, h=b, h=c, h=d, h=g$.
\end{enumerate}
\begin{proof}
The proof of the Theorem follows by Lemmas \ref{wt}-\ref{lem11}.
\end{proof}
\begin{remark}
\normalfont
We remark that the list of exceptional cases given in Theorem \ref{th} is open. In fact, weight test and curvature distribution can not be applied to these cases to prove  Levin conjecture. Therefore, some new methods need to be developed to establish the validity of Levin conjecture for this group equation of length 9.  
\end{remark}

\end{document}